\documentclass[12pt]{amsart}

\usepackage{amsmath, amssymb}
\usepackage{amscd}
\usepackage{verbatim}
\usepackage{xcolor}
\newtheorem{definition}{Definition}[section]
\newtheorem{theorem}[definition]{Theorem}
\newtheorem{lemma}[definition]{Lemma}
\newtheorem{corollary}[definition]{Corollary}
\newtheorem{proposition}[definition]{Proposition}
\theoremstyle{definition}
\newtheorem{remark}[definition]{Remark}

\newtheorem{example}[definition]{Example}

\newcommand{\la}{\left\langle}
\newcommand{\ra}{\right\rangle}

\newcommand{\clB}{\mathcal{B}}

\newcommand{\bbN}{\mathbb{N}}
\newcommand{\bbC}{\mathbb{C}}

\newcommand{\ucpa}{\text{UCP}_A(H)}
\newcommand{\ucpan}{\text{UCP}_A(\bbC^n)}

\newcommand{\ucpann}{\text{UCP}_{M_n(A)}(\bbC^n)}

\newcommand{\di}{\int^\oplus_{\ucpa}}
\newcommand{\diX}{\int^\oplus_{X}}
\newcommand{\disa}{\int^\oplus_{S(A)}}

\newcommand{\acom}{{\rho(A)}^{\prime}}

\begin{document}


\title[Generalized Orthogonal Measures]{Generalized Orthogonal Measures on the Space of Unital Completely Positive Maps}

\author[A.~Bhattacharya]{Angshuman Bhattacharya}
\address{Department of Mathematics, IISER Bhopal, MP 462066, India}
\email{angshu@iiserb.ac.in}

\author[C. J.~Kulkarni]{Chaitanya J. Kulkarni}
\address{Department of Mathematics, IISER Bhopal, MP 462066, India}
\email{kulkarni18@iiserb.ac.in}



\begin{abstract}
A classical result by Effros connects the barycentric decomposition of a state on a C*-algebra to the disintegration theory of the GNS representation of the state with respect to an orthogonal measure on the state space of the C*-algebra. In this note, we take this approach to the space of unital completely positive maps on a C*-algebra with values in $B(H)$, connecting the barycentric decomposition of the unital completely positive map and the disintegration theory of the minimal Stinespring dilation of the same. This generalizes Effros' work in the non-commutative setting. We do this by introducing a special class of barycentric measures which we call \textit{generalized orthogonal} measures.  We end this note by mentioning some examples of \textit{generalized orthogonal} measures.  
\end{abstract}

\maketitle


All C*-algebras considered in this article are unital and separable.
\section{Introduction}

In this article, we use barycentric techniques to determine when a unital completely positive map will admit a diagonal form in the decomposition theory sense. The motivation of this investigation comes from the classical result of E. Effros, which serves as a link between barycentric (integral) representation of a state $\omega$ of a C*-algebra $A$ and the disintegration of the GNS representation $\pi_\omega$ of $A$ on the GNS Hilbert space $H_\omega$. Here, we recall the classical result of E. Effros:

\begin{theorem}\cite[Theorem 4.4.9]{BR1}\label{thm;Effros}
Let $A$ be a C*-algebra and $S(A)$ be the state space of $A$ considered with the weak*-topology. Let $\omega \in S(A)$ and $\mu$ be a Borel probability measure on $S(A)$ with barycenter $\omega$, that is $\omega = \int_{S(A)} \omega^\prime \, \mathrm{d} \mu$. Then, we ge $\mu$ is an orthogonal measure (see Definition \ref{def;omsa}) if and only if $\pi_\omega = \disa \pi_{\omega^\prime} \, \mathrm{d} \mu$, where $\pi_\omega$ and $\pi_{\omega^\prime}$ are corresponding GNS representations of $\omega$ and $\omega^\prime$ respectively.
\end{theorem}

The connection between the barycentric theory and disintegration theory as demonstrated by Theorem \ref{thm;Effros} is due to a special class of representing measures (of an element) on the state space of $A$ called \textit{orthogonal measures}. To achieve our goal, we generalize the notion of orthogonal measures on the state space of C*-algebra $A$ to what we call \textit{generalized} orthogonal measures on the space of unital completely positive(UCP) maps from $A$ to $B(H)$. We fix the following notation for this article:
\begin{equation*}
\ucpa = \{ \phi : A \rightarrow B(H) ~ | ~ \phi ~ \text{is a unital completely positive map} \}.
\end{equation*}

This article is organised as follows: in Section \ref{S2}, we briefly recall some results and definitions that are required for our analysis. In Section \ref{S3}, we introduce the notion of \textit{generalized orthogonal measures} on the compact (in the BW topology) convex set $\ucpa$ as a special representing measure of an element $\phi \in \ucpa$. Then we characterize the set of all generalized orthogonal measures with barycenter $\phi$ among all representing measures of $\phi$ on $\ucpa$. This characterization leads to the generalization of the notion of orthogonal abelian subalgebras from the classical barycentric theory of state space of a C*-algebra \cite[Section 4.1]{BR1} to the $\ucpa$ context. These \textit{generalized orthogonal abelian subalgebras} reside in the commutant of the minimal Stinespring representation of $\phi \in \ucpa$ in this case. Finally, we conclude the section by proving that there is a bijective correspondence between the set of generalized orthogonal measures and generalized orthogonal abelian subalgebras. We end the article with Section \ref{S4} by giving some examples of unital completely positive maps with ranges in $M_n(\bbC)$ admitting generalized orthogonal measures as their barycentric measures among others.

\section{Preliminaries} \label{S2}
In this section, we briefly recall some results and definitions that will be used throughout the article. First and foremost, we recall the Stinespring dilation theorem. 

\begin{theorem}\cite[Theorem 4.1]{Paulsen}\label{thm;SD}
Let $A$ be a unital C*-algebra and $H$ be a Hilbert space. Let $\phi : A \rightarrow B(H)$ be a completely positive map. Then there exists a Hilbert space $K$, a bounded linear map $V : H \rightarrow K$ and a unital *-homomorphism $\rho : A \rightarrow B(K)$ such that
\begin{equation*}
\phi(a) = V^*\rho(a)V \hspace{10pt} \text{for all $a \in A$}. 
\end{equation*} 
Moreover, the set $\{ \rho(a)Vh ~|~ a \in A, h \in H \}$ spans a dense subspace of $K$.	
\end{theorem}

We call the triple $(\rho, V, K)$ from Theorem \ref{thm;SD} as the \textit{minimal} Stinespring representation for $\phi$. For a given completely positive map $\phi$, the minimal Stinespring representation is unique upto an unitary equivalence (see \cite[Proposition 4.2]{Paulsen}).

Then we recall Arveson's Radon--Nikodym theorem for completely positive maps \cite{Arveson1}. Let $A$ be a unital C*-algebra and $H$ be a Hilbert space. Let $\text{CP}(A, B(H))$ be the set of all completely positive maps from $A$ to $B(H)$. Let $\phi_1, \phi_2$ belong to $\text{CP}(A, B(H))$. Then the partial order on $\text{CP}(A, B(H))$ is given by $\phi_1 \leq \phi_2$, if $\phi_2 - \phi_1 \in \text{CP}(A, B(H))$. For $\phi \in \text{CP}(A, B(H))$ 
\begin{equation*}
[0, \phi] = \{ \theta \in \text{CP}(A, B(H)) ~|~ \theta \leq \phi \}.
\end{equation*} 
Let $V^* \rho V$ be the minimal Stinespring dilation of $\phi$ as in Theorem \ref{thm;SD}. Corresponding to each operator $T \in \rho(A)^\prime$, define $\phi_T := V^* \rho T V$ to be a map from $A$ to $B(H)$. Then we have the following theorem due to Arveson:

\begin{theorem} \cite[Theorem 1.4.2]{Arveson1} \label{thm;arnd}
The map $T \mapsto \phi_T = V^* \rho T V$ is an affine order isomorphism of the partially ordered convex set of operators $\{T \in \rho(A)^\prime : 0 \leq T \leq  1_K \}$ onto $[0, \phi]$.
\end{theorem}

In the next part of this section, we recall some definitions and results from the theory of decomposition of representations of separable C*-algebras. The reader is directed to \cite{DixC, KR2, Takesaki1} for a comprehensive introduction to this topic.

Let $(X,\nu)$ be a standard measure space and let $\{ H_p \}_{p \in X}$ denote a \textit{measurable family} of separable Hilbert spaces \cite[Definition 4.4.1B]{BR1}. The direct integral Hilbert space of the family $\{H_p\}_{p\in X}$ is denoted by 
\begin{equation*}
H=\diX H_p \, \mathrm{d} \nu.
\end{equation*}
The abelian von-Neumann algebra $\text{L}^\infty (X, \nu)$ acts as the algebra of multiplication operators on the direct integral Hilbert space $H$, and are called algebra of \textit{diagonalizable operators} on $H$. For a measurable, essentially bounded family of operators $\{T_p\}_{p\in X}$ \cite[Section 4.4.1]{BR1}, denote the direct integral of $\{T_p\}_{p\in X}$ by
\begin{equation*}
\diX T_p \, \mathrm{d} \nu.
\end{equation*}
The bounded operators of this form on the direct integral Hilbert space $H$ are called \textit{decomposable operators}. The collection of all decomposable operators forms a von-Neumann algebra and is called the algebra of decomposable operators.

Let $A$ be a unital separable C*-algebra and let $\{\rho_p : A \rightarrow B(H_p)\}_{p\in X}$ be a family of representations of $A$ on the measurable family of Hilbert spaces $\{H_p\}_{p\in X}$. The family $\{\rho_p\}_{p\in X}$ is called measurable if for all $a \in A$, the family $\{\rho_p(a)\}_{p\in X}$ is a measurable family of essentially bounded operators. Therefore, for all $a \in A$: 
\begin{equation*}
\diX \rho_p (a) \, \mathrm{d} \nu
\end{equation*}
is a decomposable operator.

Now, let $A$ be a separable unital C*-algebra and $\rho : A \rightarrow B(H)$ be a non-degenerate representation of $A$ on a separable Hilbert space $H$. Let $\acom$ denote the commutant of $\rho(A)$ in $B(H)$. Let $\clB \subseteq \acom$ be an abelian von-Neumann subalgebra of $\acom$. The direct integral decomposition of $\rho$ with respect to $\clB$ is given by the following fundamental theorem for the spatial decomposition of representations:

\begin{theorem}\cite[Theorem 4.4.7]{BR1} \label{thm;dint}
Let $A$ be a separable C*-algebra, $\rho$ be a non-degenerate representation of $A$ on a separable Hilbert space $H$, and $\clB$ be an abelian von-Neumann subalgebra of $\acom$. It follows that there exists a standard measure space $X$, a positive bounded measure $\nu$ on $X$, a measurable family of Hilbert spaces $\{H_p\}_{p\in X}$, a measurable family of representations $\{\rho_p\}_{p\in X}$ on $\{H_p\}_{p\in X}$ and a unitary map 
\begin{equation*}
U : H \rightarrow \diX H_p \, \mathrm{d} \nu
\end{equation*}
such that $U \clB U^*$ is just the set of diagonalizable operators on $\diX H_p \, \mathrm{d} \nu$ and 
\begin{equation*}
U \rho(a) U^* = \diX \rho_p (a) \, \mathrm{d} \nu
\end{equation*} 
for all $a \in A$.
\end{theorem}

We end this section by recalling some definitions from the classical case of orthogonality of the measure on the state space of a C*-algebra.

\begin{definition}\cite[Definition 4.1.20]{BR1}\label{def;op}
Let $\omega_1, \omega_2$ be two positive linear functionals over the C*-algebra $A$ such that $\omega_1 + \omega_2 = \omega$ and  $(\pi_\omega, H_\omega, \Omega_\omega)$ be the GNS-triple corresponding to $\omega$. Then, we say $\omega_1$ is orthogonal to $\omega_2$, if there exists a projection $P \in \pi_\omega(A)^\prime$ such that 
\begin{equation*}
\omega_1(a) = \la \pi_\omega(a) \Omega_\omega, P \Omega_\omega \ra \hspace{10pt} \text{and} \hspace{10pt} \omega_2(a) = \la \pi_\omega(a) \Omega_\omega, (1_{H_\omega} - P) \Omega_\omega \ra.
\end{equation*}
\end{definition}

Let the state space of a C*-algebra $A$ be denoted by $S(A)$. Consider the set $S(A)$ with the weak*-topology. 

\begin{definition}\cite[Definition 4.1.20]{BR1}\label{def;omsa}
Let $\omega$ be a state on $A$ and $\mu$ be a measure with berycenter $\omega$ that is $\omega = \int_{S(A)} \omega^\prime \, \mathrm{d} \mu$. Then, we say that $\mu$ is an orthogonal measure, if for every Borel measurable subset $E$ of $S(A)$, we have 
$\int\limits_{E} \omega^\prime \, \mathrm{d} \mu$ is orthogonal to $\int\limits_{S(A) \setminus E} \omega^\prime \, \mathrm{d} \mu$.
\end{definition}

\section{Generalised Orthogonal Measures} \label{S3}
Let $A$ be a unital, separable C*-algebra and $H$ be a separable Hilbert space. Recall $\ucpa$ is the set of all unital completely positive maps from $A$ to $B(H)$. This set is a compact, convex subset of $\text{CB}(A, B(H))$ with respect to the BW topology. Here, $\text{CB}(A, B(H))$ denotes the set of all completely bounded maps from $A$ to $B(H)$. Suppose $M_1(\ucpa)$ denotes the set of all positive Borel measures on $\ucpa$ with norm 1. For $\phi \in \ucpa$, we define the set
\begin{equation*}
M_\phi(\ucpa) := \left \{ \mu \in M_1(\ucpa) \, \, | \, \, \phi = \int_{\ucpa} \phi^\prime \, \mathrm{d} \mu \right \}
\end{equation*}
where the integral is taken in a weak sense. That is, for all $a \in A$ and $h_1, h_2 \in H$, we have $\la \phi(a) h_1, h_2 \ra = \int_{\ucpa} \la \phi^\prime(a)h_1, h_2 \ra  \, \mathrm{d} \mu$. If this happens, then we say that the $\phi$ is the barycenter of $\mu$. Now, we define the notion of orthogonality between two completely positive maps from $A$ to $B(H)$. 

\begin{definition}\label{def;ocp}
Let $\phi_1, \phi_2$ be two completely positive maps from $A$ to $B(H)$ such that $\phi_1 + \phi_2 = \phi$ and  $V^* \rho V$ be the minimal Stinespring dilation of $\phi$. Then, we say $\phi_1$ is orthogonal to $\phi_2$ and denoted by $\phi_1 \perp \phi_2$, if there exists a projection $P \in \acom$ such that 
\begin{equation*}
\phi_1 =  V^* P \rho V  \hspace{10pt} \text{and} \hspace{10pt} \phi_2 =  V^* (1-P) \rho V.
\end{equation*}
\end{definition}

Using the above definition, for fixed $\phi \in \ucpa$, we state the definition of a generalized orthogonal measure.

\begin{definition}\label{def;om}
Let $\mu \in M_\phi(\ucpa)$. Then, we say that the measure $\mu$ is a generalized orthogonal measure if for every Borel measurable subset $E$ of $\ucpa$ we get 
\begin{equation*}
\int\limits_{E} \phi^\prime \, \mathrm{d} \mu \perp \int\limits_{\ucpa \setminus E} \phi^\prime \, \mathrm{d} \mu.
\end{equation*}
\end{definition}

\begin{remark}
Hereafter in this article, a generalized orthogonal measure will be referred to as an orthogonal measure (on $\ucpa$) for ease of readability without any ambiguity.
\end{remark}

We denote the set of all measures that are orthogonal and with barycenter $\phi$ by $\mathcal{O}_\phi(\ucpa)$. We have the following containment between the three sets:
\begin{equation*}
\mathcal{O}_\phi(\ucpa) \subseteq M_\phi(\ucpa) \subseteq M_1(\ucpa).
\end{equation*}

Let $\phi$ be in $\ucpa$ and $V^* \rho V$ be the minimal Stinespring dilation of $\phi$. For $\mu \in M_\phi(\ucpa)$, the following lemma gives a unique map from $\text{L}^\infty(\ucpa, \mu)$ to $\acom$ satisfying certain conditions. On $\text{L}^\infty(\ucpa, \mu)$, we consider $\sigma(\text{L}^\infty, \text{L}^1)$ topology, that is, $f_\alpha \rightarrow f$ in $\sigma(\text{L}^\infty, \text{L}^1)$ topology if and only if
\begin{equation*}
\int\limits_{\ucpa} f_\alpha g \, \mathrm{d} \mu \rightarrow \int\limits_{\ucpa} f g \, \mathrm{d} \mu
\end{equation*}
for all $g \in \text{L}^1(\ucpa, \mu)$. While $\acom$ is considered with the weak operator topology (WOT).

\begin{lemma}\label{lem;kmu}
Let $\mu \in M_\phi(\ucpa)$ where $V^* \rho V$ be the minimal Stinespring dilation of $\phi$. Then there exists a unique map 
\begin{equation*}
k_\mu : \text{L}^\infty(\ucpa, \mu) \rightarrow \acom
\end{equation*}
given by 
\begin{equation*}
\la k_\mu(f) \rho(a) Vh_1, \rho(b) Vh_2 \ra = \int\limits_{\ucpa} f(\phi^\prime) \la \phi^\prime (b^*a)h_1, h_2 \ra  \, \mathrm{d} \mu
\end{equation*}
where $a, b \in A$ and $h_1, h_2 \in H$. The map is positive, contractive and continuous in $\sigma(\text{L}^\infty, \text{L}^1) - \text{WOT}$. 
\end{lemma}
\begin{proof}
For a given positive function $f$ in $\text{L}^\infty(\ucpa, \mu)$, define a completely positive map $\phi_f : A \rightarrow B(H)$ as $\phi_f := \int_{\ucpa} f(\phi^\prime) \phi^\prime \, \mathrm{d} \mu$. For $a \in A$ and $h_1, h_2 \in H$ the integral is observed as 
\begin{equation*}
\la \phi_f(a)h_1, h_2 \ra = \int\limits_{\ucpa} f(\phi^\prime) \la \phi^\prime(a)h_1, h_2 \ra \, \mathrm{d} \mu.
\end{equation*}
If $f \geq 0$ and $\| f \|_\infty \leq 1$, then we get
\begin{equation*}
\phi_f = \int\limits_{\ucpa} f(\phi^\prime) \phi^\prime \, \mathrm{d} \mu \leq \int\limits_{\ucpa} \phi^\prime \, \mathrm{d} \mu = \phi.
\end{equation*}
By using Theorem \ref{thm;arnd}, corresponding to $\phi_f$, we get a unique positive operator, say $k_\mu(f) \in \acom$ such that $0 \leq k_\mu(f) \leq 1_K$ and $\phi_f = V^*k_\mu(f) \rho V$. Therefore, for all $a, b \in A$ and $h_1, h_2 \in H$ we have,
\begin{align*}
\la k_\mu(f) \rho(a) V h_1, \rho(b)Vh_2 \ra &=	\la \phi_f(b^*a) h_1, h_2 \ra &\\
&= \int\limits_{\ucpa} f(\phi^\prime) \la \phi^\prime(b^*a)h_1, h_2 \ra \, \mathrm{d} \mu.
\end{align*}
Then we extend the definition of $\phi_f$ for all $f\in \text{L}^\infty(\ucpa, \mu)$. For this, we consider $f \in \text{L}^\infty(\ucpa, \mu)$ as a linear combination of four positive elements in $\text{L}^\infty(\ucpa, \mu)$. Suppose $f = \sum\limits_{i=0}^{3} c_if_i = \sum\limits_{i=0}^{3} d_ig_i$, where $f_i$ and $g_i$ are positive functions and $c_i$ and $d_i$ are scalers, then 
\begin{flalign*}
\la \sum\limits_{i=0}^{3} c_ik_\mu(f_i) \rho(a) V h_1, \rho(b)Vh_2 \ra &= \int\limits_{\ucpa} \sum\limits_{i=0}^{3} c_if_i(\phi^\prime) \la \phi^\prime(b^*a)h_1, h_2 \ra & \\
&= \int\limits_{\ucpa} \sum\limits_{i=0}^{3} d_ig_i(\phi^\prime) \la \phi^\prime(b^*a)h_1, h_2 \ra & \\
&= \la \sum\limits_{i=0}^{3} d_ik_\mu(g_i) \rho(a) V h_1, \rho(b)Vh_2 \ra
\end{flalign*} 
Therefore, $\sum\limits_{i=0}^{3} c_ik_\mu(f_i) = \sum\limits_{i=0}^{3} d_ik_\mu(g_i)$ and we get the existence of $k_\mu(f)$ for all $f \in \text{L}^\infty(\ucpa, \mu)$. The map $k_\mu$ is a unital and positive map from $\text{L}^\infty(\ucpa, \mu)$ into $\acom$. As the domain is commutative, the map $k_\mu$ is completely positive. Therefore, $\|k_\mu\| = \|k_\mu(1)\| = 1$ that implies $k_\mu$ is contractive. If $f_\alpha \rightarrow f$ in $\sigma(\text{L}^\infty, \text{L}^1)$ topology, then this implies for all $a, b \in A$ and $h_1, h_2 \in H$
\begin{equation*}
\int\limits_{\ucpa} f_\alpha(\phi^\prime) \la \phi^\prime(b^*a)h_1, h_2 \ra \, \mathrm{d} \mu \rightarrow \int\limits_{\ucpa} f(\phi^\prime) \la \phi^\prime(b^*a)h_1, h_2 \ra \, \mathrm{d} \mu.
\end{equation*}
That is $\la k_\mu(f_\alpha) \rho(a) V h_1, \rho(b)Vh_2 \ra \rightarrow \la k_\mu(f) \rho(a) V h_1, \rho(b)Vh_2 \ra$. Since $K$ is the closed linear span of $\{ \rho(a)Vh ~|~ a \in A, h\in H \}$, we get $k_\mu(f_\alpha) \rightarrow k_\mu(f)$ in WOT. Therefore, the map $k_\mu : \text{L}^\infty(\ucpa, \mu) \rightarrow \acom$ is continuous in $\sigma(\text{L}^\infty, \text{L}^1) - \text{WOT}$. 
\end{proof}

For fixed $\phi \in \ucpa$, the following proposition characterizes the orthogonal measures $\mu$ belonging to $M_\phi(\ucpa)$ by using the map $k_\mu$ defined in Lemma \ref{lem;kmu}.
\begin{proposition}\label{prop;om}
Let $\mu \in M_\phi(\ucpa)$ and $V^* \rho V$ be the minimal Stinespring dilation of $\phi$. Then the following statements are equivalent
\begin{enumerate}
\item The measure $\mu$ is in $\mathcal{O}_\phi(\ucpa)$.
\item The map $f \mapsto k_\mu(f)$ is a *-isomorphism from $\text{L}^\infty(\ucpa, \mu)$ onto the range of $k_\mu$ in $\acom$.
\item The map $f \mapsto k_\mu(f)$ is a *-homomorphism.
\end{enumerate}
\end{proposition}
\begin{proof}
Assume $\mu$ is in $\mathcal{O}_\phi(\ucpa)$. By using Lemma \ref{lem;kmu}, we know that the map $f \mapsto k_\mu(f)$ is linear and positive . If $f$ is a projection, then there exists a measurable set, say $E \subseteq \ucpa$ such that $f = \chi_E$, where $\chi_E$ is the characteristic function of $E$. Since $\mu$ is orthogonal, we have 
\begin{equation*}
\int_{E} \phi^\prime \, \mathrm{d} \mu \perp \int_{\ucpa \setminus E} \phi^\prime \, \mathrm{d} \mu
\end{equation*}
Hence, by Definition \ref{def;ocp}, we have $k_\mu(f)$ to be a projection. If $f$ and $g$ are orthogonal projections, then $f \leq 1-g$. Hence $k_\mu(f) \leq 1_K - k_\mu(g)$ and $k_\mu(f)k_\mu(g) = 0$. If $f$ and $g$ are arbitrary projections in $\text{L}^\infty(\ucpa, \mu)$, then each of the pairs $\{ f(1-g), fg\}$, $\{fg, (1-f)g \}$ and $\{f(1-g), (1-f)g \}$ is orthogonal. Thus the decomposition $f = fg + f(1-g)$ and $g = gf + g(1-f)$ implies that $k_\mu(fg) = k_\mu(f)k_\mu(g)$.
	
Now any elements $f, g \in \text{L}^\infty(\ucpa, \mu)$ can be approximated in norm by linear combinations of projections and we have $\|k_\mu(f)\| \leq \|f\|_\infty$. This implies the relation $k_\mu(fg) = k_\mu(f)k_\mu(g)$ holds true for all $f$ and $g$ in $\text{L}^\infty(\ucpa, \mu)$. Therefore, $k_\mu$ is a *-homomorphism. Now we show that $k_\mu$ is faithful. If $f \neq 0$, then we have for $h \in H$ with $\|h\| = 1$ 
\begin{align*}
\la k_\mu(f) V h, k_\mu(f) Vh \ra &= \la k_\mu(\bar{f}f) V h, Vh \ra & \\
&= \int\limits_{\ucpa} \bar{f}f(\phi^\prime) \la \phi^\prime(1_A)h, h \ra \, \mathrm{d} \mu &\\
&= \int\limits_{\ucpa} |f(\phi^\prime)|^2 \la h, h \ra \, \mathrm{d} \mu > 0.
\end{align*}
Hence we get (1) $\implies$ (2). 
	
(2) $\implies$ (3) follows trivially.

Now we assume $f \mapsto k_\mu(f)$ is a *-homomorphism. If $E \subseteq \ucpa$ is a Borel measurable set, then by Lemma \ref{lem;kmu} we get unique elements $k_\mu(\chi_E)$ and $k_\mu(\chi_{\ucpa \setminus E})$ of $\acom$ such that 
\begin{equation*}
\int\limits_{E} \phi^\prime \, \mathrm{d} \mu = V^*k_\mu(\chi_E)\rho V \hspace{7pt} \text{and} \hspace{7pt} \int\limits_{\ucpa \setminus E} \phi^\prime \, \mathrm{d} \mu = V^*k_\mu(\chi_{\ucpa \setminus E})\rho V.
\end{equation*}
Since $k_\mu$ is a *-homomorphism, $k_\mu(\chi_E)$ and $k_\mu(\chi_{\ucpa \setminus E})$ are mutually orthogonal projections of sum 1. Hence, by Definition \ref{def;ocp} we get 
\begin{equation*}
\int\limits_{E} \phi^\prime \, \mathrm{d} \mu \perp \int\limits_{\ucpa \setminus E} \phi^\prime \, \mathrm{d} \mu
\end{equation*}
Therefore, $\mu$ is orthogonal and we have (3) $\implies$ (1). 
\end{proof}

\begin{remark}\label{rem;avna}
If one of the conditions in the previous propositions is satisfied, then $\clB_\mu = \{ k_\mu(f) ~|~ f \in \text{L}^\infty(\ucpa, \mu) \}$ is an abelian von Neumann subalgebra of $\acom$.
\end{remark}

Let $\phi \in \ucpa$ and suppose there exists a probability measure $\mu$ on $\ucpa$ such that $\phi = \int_{\ucpa} \phi^\prime  \, \mathrm{d} \mu$. That is $\mu \in M_\phi(\ucpa)$. Let $V^*\rho V$ and $V_{\phi^\prime}^*\rho_{\phi^\prime} V_{\phi^\prime}$ be the minimal Stinespring dilations of $\phi$ and $\phi^\prime$ respectively, where $V : H \rightarrow K$ and $V_{\phi^\prime} : H \rightarrow K_{\phi^\prime}$ are isometries. We show that the family $\{K_{\phi^\prime}\}_{\ucpa}$ is a measurable family of Hilbert spaces. 

Let $\{a_n\}_{n \geq 1}$ and $\{h_m\}_{m \geq 1}$ be countable dense subsets of $A$ and $H$ respectively. Consider a sequence $(f_{n m})$ of functions such that $f_{n m}(\phi^\prime) := \rho_{\phi^\prime}(a_n)V_{\phi^\prime}h_m$. The function $\phi^\prime \mapsto \la f_{i j}(\phi^\prime), f_{k l}(\phi^\prime) \ra$ is measurable for all $i, j, k, l \geq 1$. Because, we have $\la f_{i j}(\phi^\prime), f_{k l}(\phi^\prime) \ra = \la \phi^\prime(a^*_ka_i)h_j, h_l \ra$ and the function $\phi^\prime \mapsto \la \phi^\prime(a^*_ka_i)h_j, h_l \ra$ is a measurable function. Also from the minimality of $K_{\phi^\prime}$ we have the set $\{f_{n m}(\phi^\prime) ~|~ n, m \geq 1 \}$ is dense in $K_{\phi^\prime}$ for almost every $\phi^\prime$. Hence $\{K_{\phi^\prime}\}_{\ucpa}$ is a measurable family of Hilbert spaces. Now consider a Hilbert space $K_\mu := \di K_{\phi^\prime} \, \mathrm{d} \mu$. For every $f \in \text{L}^\infty(\ucpa, \mu)$, we have a diagonalizable operator, say $D_f$ on $K_\mu$. The map $f \mapsto D_f$ is a *-isomorphism between $\text{L}^\infty(\ucpa, \mu)$ onto the algebra of all diagonalizable operators on $K_\mu$. 

Define a representation $\rho_\mu := \di \rho_{\phi^\prime} \, \mathrm{d} \mu$ of $A$ on $K_\mu$. The representation $\rho_\mu$ is defined as $\rho_\mu(a) := \di \rho_{\phi^\prime}(a) \, \mathrm{d} \mu$ for all $a \in A$. For defining $\rho_\mu$, first we have to check the measurability of $\rho_\mu(a)$ for all $a \in A$. We can observe that $\rho_\mu(a)$ is a measurable operator on $K_\mu$ because, $\la \rho_{\phi^\prime}(a) f_{i j}(\phi^\prime), f_{k l}(\phi^\prime) \ra = \la \phi^\prime(a^*_kaa_i)h_j, h_l \ra$ for all $i, j, k, l \geq 1$ and the function $\phi^\prime \mapsto \la \phi^\prime(a^*_kaa_i)h_j, h_l \ra$ is a measurable function.
 
Now for $h \in H$, consider $h^\mu := \di V_{\phi^\prime}h \, \mathrm{d} \mu$. Since the function $\phi^\prime \mapsto \la f_{n m}(\phi^\prime), V_{\phi^\prime}h \ra = \la \phi^\prime(a_n)h_m, h \ra$ is a measurable function, we get that $h^\mu \in K_\mu$. We have for all $a \in A$ and $h \in H$,
\begin{align*}
\la \rho_\mu(a) h^\mu, h^\mu \ra &= \la \di \rho_{\phi^\prime}(a)V_{\phi^\prime}h \, \mathrm{d} \mu , \di V_{\phi^\prime}h \, \mathrm{d} \mu \ra & \\
&= \int_{\ucpa} \la V^*_{\phi^\prime} \rho_{\phi^\prime}(a) V_{\phi^\prime}h, h \ra \, \mathrm{d} \mu & \\
&= \int_{\ucpa} \la \phi^\prime(a) h, h \ra \, \mathrm{d} \mu & \\
&= \la \phi(a)h, h \ra = \la \rho(a)Vh, Vh \ra.
\end{align*}

Using the fact that the Hilbert space $K$ is the closed linear span of $\{ \rho(a)Vh ~|~ a \in A, ~ h\in H \}$, define a map $U_\mu : K \rightarrow K_\mu$ as $ U_\mu(\rho(a)Vh) := \rho_\mu(a)h^\mu$. First, we see that the map $U_\mu$ is well defined and isometric. For this, observe for $a, b \in A$ and $h_1, h_2 \in H$ we have
\begin{align*}
\la U_\mu(\rho(a)Vh_1), U_\mu(\rho(b)Vh_2) \ra &= \la \rho_\mu(a)h^\mu_{1}, \rho_\mu(b)h^\mu_{2} \ra & \\
&= \int_{\ucpa} \la V_{\phi^\prime}^*\rho_{\phi^\prime}(b^*a)V_{\phi^\prime}h_1, h_2 \ra \, \mathrm{d} \mu &\\
&= \int_{\ucpa} \la \phi^\prime(b^*a)h_1, h_2 \ra \, \mathrm{d} \mu &\\
&= \la \phi(b^*a)h_1, h_2 \ra &\\
&= \la \rho(a)Vh_1, \rho(b)Vh_2 \ra.
\end{align*}
Therefore, we get for all $a, b \in A$ and $h_1, h_2 \in H$
\begin{equation*}
\la U_\mu(\rho(a)Vh_1), U_\mu(\rho(b)Vh_2) \ra = \la \rho(a)Vh_1, \rho(b)Vh_2 \ra.
\end{equation*}

Let $\{a_r\}^n_{r=1}$, $\{a_s\}^m_{s=1}$, $\{h_r\}^n_{r=1}$, $\{h_s\}^m_{s=1}$ be such that $a_r, a_s \in A$ and $h_r, h_s \in H$ for all $r$ and $s$. Then for arbitrary $\sum\limits_{r=1}^{n}\rho(a_r)Vh_r$ and $\sum\limits_{s=1}^{m}\rho(a_s)Vh_s$ in the Hilbert space $K$, we have the following 
\begin{flalign*}
\la \sum_{r=1}^{n} \rho_\mu(a_r)h_r^\mu, \sum_{s=1}^{m} \rho_\mu(a_s)h_s^\mu \ra
&= \sum_{r=1, s=1}^{n, m} \la U_\mu (\rho(a_r)Vh_r), U_\mu (\rho(a_s)Vh_s) \ra & \\
&= \sum_{r=1, s=1}^{n, m}  \la \rho(a_r)Vh_r, \rho(a_s)Vh_s \ra & \\
&= \la \sum_{r=1}^{n}\rho(a_r)Vh_r, \sum_{s=1}^{m}\rho(a_s)Vh_s \ra
\end{flalign*}
This shows that $\la U_\mu \left (\sum_{r=1}^{n}\rho(a_r)Vh_r \right ), U_\mu \left (\sum_{s=1}^{m}\rho(a_s)Vh_s \right ) \ra$ is equal to  $\la \sum_{r=1}^{n}\rho(a_r)Vh_r, \sum_{s=1}^{m}\rho(a_s)Vh_s \ra$ for arbitrary elements $\sum\limits_{r=1}^{n}\rho(a_r)Vh_r$ and $\sum\limits_{s=1}^{m}\rho(a_s)Vh_s$  in the Hilbert space $K$. This proves that the map $U_\mu$ is well defined and isometric on the closed linear span of the elements $\{ \rho(a)Vh ~|~ a \in A, ~ h\in H \}$. That is $U_\mu : K \rightarrow K_\mu$ is an isometry.

The map $U_\mu$ constructed above with respect to the measure $\mu$ in $M_\phi(\ucpa)$ is used in characterizing $\mu \in \mathcal{O}_\phi(\ucpa).$ This has been done in Theorem \ref{thm;unitary}, which states that $\mu \in \mathcal{O}_\phi(\ucpa)$ if and only if the corresponding isometry $U_\mu$ is unitary. To prove Theorem \ref{thm;unitary}, first, we need the following lemma. 

\begin{lemma}\label{lem;dense}
The set
\begin{equation*}
K^\circ_\mu := \{D_f \rho_\mu(a)h^\mu ~|~  f \in \text{L}^\infty(\ucpa, \mu), ~ a \in A, ~ h \in H \}
\end{equation*}
has the dense linear span in $K_\mu$.
\end{lemma} 
\begin{proof}
If there exists $\eta \in (K^\circ_\mu)^\perp$, then $\int_{\ucpa} \la f(\phi^\prime) \rho_{\phi^\prime}(a)V_{\phi^\prime}h, \eta_{\phi^\prime} \ra  \, \mathrm{d} \mu$ is 0. That is $\int_{\ucpa} f(\phi^\prime)  \la \rho_{\phi^\prime}(a)V_{\phi^\prime}h, \eta_{\phi^\prime} \ra  \, \mathrm{d} \mu = 0$. Since this is true for all $f \in \text{L}^\infty(\ucpa, \mu)$, we get $\la \rho_{\phi^\prime}(a)V_{\phi^\prime}h, \eta_{\phi^\prime} \ra = 0$ almost everywhere. Here, observe that the Hilbert space $K_{\phi^\prime}$ is the closed linear span of $\{ \rho_{\phi^\prime}(a)V_{\phi^\prime}h ~|~ a \in A, ~ h\in H \}$. Also, $a \in A$, $h \in H$ were arbitrary and the separability of $A$ and $H$ imply that $\eta_{\phi^\prime} = 0$ almost everywhere. Hence $\eta = 0$. Therefore, the set $K^\circ_\mu$ has the dense linear span in $K_\mu$.
\end{proof}

\begin{theorem} \label{thm;unitary}
Let $\mu \in M_\phi(\ucpa)$. Then $\mu$ is in $\mathcal{O}_\phi(\ucpa)$ if and only if the corresponding operator $U_\mu$ is unitary.
\end{theorem}
\begin{proof}
For a given $\phi \in \ucpa$ and $\mu \in M_\phi(\ucpa)$, we construct the operator $U_\mu$ as discussed above. Moreover, we know that $U_\mu$ is an isometry. We assume that $U_\mu$ is unitary and suppose $V^* \rho V$ be the minimal Stinespring dilation of $\phi$. Then for $f \in \text{L}^\infty(\ucpa, \mu)$, $a, b \in A$ and $h_1, h_2 \in H$ we have,
\begin{align*}
\la U_\mu^*D_f U_\mu (\rho(a)Vh_1), \rho(b)Vh_2 \ra &= \la D_f\rho_\mu(a)h^\mu_{1}, \rho_\mu(b)h^\mu_{2} \ra &\\
&= \int\limits_{\ucpa} f(\phi^\prime) \la \rho_{\phi^\prime}(b^*a)V_{\phi^\prime}h_1, V_{\phi^\prime}h_2 \ra &\\
&= \int\limits_{\ucpa} f(\phi^\prime) \la \phi^\prime(b^*a)h_1, h_2 \ra  \mathrm{d} \mu &\\
&= \la k_\mu(f) \rho(b^*a)Vh_{1}, Vh_{2} \ra  &\\
&= \la k_\mu(f) \rho(a)Vh_{1}, \rho(b)Vh_{2} \ra	
\end{align*}
Since, the Hilbert space $K$ is the closed linear span of the elements of the form $\{ \rho(a)Vh ~|~ a \in A, ~ h\in H \}$, we get $U_\mu^*D_f U_\mu = k_\mu(f)$. As the map $f \mapsto D_f$ is a *-isomorphism and $U_\mu$ is unitary imply that the map $f \mapsto k_\mu(f)$ is a *-isomorphism from  $\text{L}^\infty(\ucpa, \mu)$ into $\acom$. Then by using Proposition~\ref{prop;om}, we get that the measure $\mu$ is an orthogonal measure with barycenter $\phi$, that is, $\mu \in \mathcal{O}_\phi(\ucpa)$. 

Converely, if $\mu \in \mathcal{O}_\phi(\ucpa)$, then we show that the corresponding isometry $U_\mu$ is an unitary operator. Since, $U_\mu : K \rightarrow K_\mu$ is an isometry to prove the converse, we need to show that the map $U_\mu$ is surjective. We prove, for given $\epsilon > 0$, $f \in \text{L}^\infty(\ucpa, \mu)$, $a \in A$ and $h \in H$, the existence of $\{a_i\}_{i = 1}^n$ and $\{h_i\}_{i = 1}^n$, where $a_i \in A$ and $h_i \in H$ such that $\|D_f \rho_\mu(a)h^\mu - \sum_{i=1}^{n} \rho_\mu(a_i)h^\mu_i \| < \epsilon$. For simplicity of the computations we denote the term $\|D_f \rho_\mu(a)h^\mu - \sum_{i=1}^{n} \rho_\mu(a_i)h^\mu_i \|$ by $N$. Then we have 
\begin{flalign*}
N^2 = & \la D_{\bar{f}f} \rho_\mu(a^*a)h^\mu, h^\mu \ra - \sum_{i =1}^{n} \la D_f \rho_\mu(a^*_{i}a)h^\mu, h_i^\mu \ra & \\
& - \sum_{i=1}^{n} \la D_{\bar{f}}\rho_\mu(a^*a_i)h_i^\mu, h^\mu \ra + \sum_{i, j=1}^{n} \la \rho_\mu(a^*_{j}a_i)h_i^\mu,  h_j^\mu \ra & \\
= &\int\limits_{\ucpa} |f(\phi^\prime)|^2 \la \rho_{\phi^\prime}(a^*a)V_{\phi^\prime}h, V_{\phi^\prime}h \ra \mathrm{d} \mu & \\
& -\sum_{i=1}^{n} \int\limits_{\ucpa} f(\phi^\prime) \la \rho_{\phi^\prime}(a^*_{i}a)V_{\phi^\prime}h, V_{\phi^\prime}h_i \ra \mathrm{d} \mu & \\
& - \sum_{i=1}^{n} \int\limits_{\ucpa} \bar{f}(\phi^\prime) \la \rho_{\phi^\prime}(a^*a_i)V_{\phi^\prime}h_i, V_{\phi^\prime}h \ra \mathrm{d} \mu
\end{flalign*}

\begin{flalign*}
& + \sum_{i,j=1}^{n} \int\limits_{\ucpa} \la \rho_{\phi^\prime}(a^*_{j}a_i)V_{\phi^\prime}h_i, V_{\phi^\prime}h_j \ra &\\
=&  \int\limits_{\ucpa} |f(\phi^\prime)|^2 \la \phi^\prime(a^*a)h, h \ra  \mathrm{d} \mu - \sum_{i=1}^{n} \int\limits_{\ucpa} f(\phi^\prime) \la \phi^\prime(a^*_{i}a)h, h_i \ra \mathrm{d} \mu  &\\
& - \sum_{i=1}^{n} \int\limits_{\ucpa} \bar{f}(\phi^\prime) \la \phi^\prime(a^*a_i)h_i, h \ra \mathrm{d} \mu + \sum_{i,j=1}^{n} \int\limits_{\ucpa} \la \phi^\prime(a^*_{j}a_i)h_i, h_j \ra \mathrm{d} \mu &\\
= & \la k_\mu(\bar{f}f) \rho(a^*a) Vh, Vh \ra - \sum_{i=1}^{n} \la k_\mu(f) \rho(a^*_{i}a) Vh, Vh_i \ra &\\
& - \sum_{i=1}^{n} \la k_\mu(\bar{f}) \rho(a^*a_i) Vh_i, Vh \ra + \sum_{i,j=1}^{n} \la \rho(a^*_{j}a_i) Vh_i, Vh_j \ra & \\
= & \la k_\mu(f) \rho(a) Vh, k_\mu(f)\rho(a)Vh \ra - \sum_{i=1}^{n} \la k_\mu(f) \rho(a) Vh, \rho(a_i)Vh_i \ra &\\ 
& - \sum_{i=1}^{n} \la \rho(a_i) Vh_i, k_\mu(f) \rho(a) Vh \ra + \sum_{i,j=1}^{n} \la \rho(a_i) Vh_i, \rho(a_j)Vh_j \ra &\\
=& \left \|k_\mu(f) \rho(a) Vh - \sum_{i=1}^{n} \rho(a_i) Vh_i \right \|^2.	
\end{flalign*}
The second last step follows because the measure $\mu$ is orthogonal. 

The minimality of $K$ implies that for given $\epsilon > 0$, a measurable function $f \in \text{L}^\infty(\ucpa, \mu)$, $a \in A$ and $h \in H$, there exists $\{a_i\}_{i = 1}^n$ and $\{h_i\}_{i = 1}^n$ where $a_i \in A$ and $h_i \in H$ such that 
\begin{equation*}
\left \|k_\mu(f) \rho(a) Vh - \sum_{i=1}^{n} \rho(a_i) Vh_i \right \|^2 < \epsilon^2.
\end{equation*} 
Therefore, 
\begin{equation*}
\left \|D_f \rho_\mu(a)h^\mu - \sum_{i=1}^{n} \rho_\mu(a_i) h_i^\mu \right \|^2 < \epsilon^2.
\end{equation*}

The element $D_f \rho_\mu(a)h^\mu \in K^\circ_\mu$ and the element $\sum_{i=1}^{n} \rho_\mu(a_i)h_i^\mu \in U_\mu K$. Then by using Lemma \ref{lem;dense}, we know that $K^\circ_\mu$ has the dense linear span in $K_\mu$. This implies $U_\mu K$ is dense in $K_\mu$. Hence, $U_\mu$ is unitary.
\end{proof}

Let $\mu$ be an element of $M_\phi(\ucpa)$, where $\phi \in \ucpa$. Suppose $V^* \rho V$ is the minimal Stinespring dilation of $\phi$, where $V : H \rightarrow K$ and $\rho : A \rightarrow B(K)$. Let $\mu$ be an orthogonal measure, then from Remark \ref{rem;avna}, we get the existence of an abelian subalgebra say $\clB_\mu$ of $\acom$. The following corollary identifies $K$ and $\rho$ with $K_\mu$ and $\rho_\mu$ respectively as the decomposition with respect to the abelian subalgebra $\clB_\mu$.

\begin{corollary}\label{cor;dint}
With the notations as in Theorem \ref{thm;unitary} we get, if the measure $\mu$ is orthogonal, then with respect to the abelian subalgebra $\clB_\mu = \{ k_\mu(f) ~|~ f \in \text{L}^\infty(\ucpa, \mu) \}$ of $\acom$ the Hilbert space $K$ and the representation $\rho$ disintegrate as $K_\mu = \di K_{\phi^\prime} \, \mathrm{d} \mu$ and $\rho_\mu = \di \rho_{\phi^\prime} \, \mathrm{d} \mu$ respectively. Moreover, the algebra of all diagonalizable operators on $K_\mu$ is given by $U_\mu \clB_\mu U_\mu^*$.
\end{corollary}
\begin{proof}
By using Theorem \ref{thm;unitary}, we get for $\phi \in \ucpa$ the measure $\mu$ belongs to $\mathcal{O}_\phi(\ucpa)$ if and only if the corresponding operator $U_\mu : K \rightarrow K_\mu$ is unitary. Suppose $\mu$ is orthogonal, then we have for all $a_1, a_2 \in A$ and $h \in H$, 
\begin{equation*} 
U_\mu^* \rho_\mu(a_1) U_\mu (\rho(a_2)Vh) = U_\mu^* \rho_\mu(a_1) \rho_\mu(a_2)h^\mu = \rho(a_1)(\rho(a_2) V h).
\end{equation*}
Then using the minimality of $K$, we get $U_\mu^* \rho_\mu(a) U_\mu = \rho(a)$ for all $a \in A$. This implies that the Hilbert spaces $K$ and $K_\mu$ are isomorphic and the representations $\rho$ and $\rho_\mu$ are unitarily equivalent. 

Now recall from the first part of the proof of Theorem \ref{thm;unitary}, that for $f \in \text{L}^\infty(\ucpa, \mu)$ we have $U_\mu^*D_f U_\mu = k_\mu(f)$. Since $\mu$ is orthogonal, $U_\mu$ is unitary which implies $U_\mu k_\mu(f) U_\mu^* = D_f$. This shows that the algebra of all diagonalizable operators on $K_\mu$ is given by $U_\mu \clB_\mu U_\mu^*$
\end{proof}

Now we define a special class of abelian subalgebras called orthogonal abelian subalgebras.
\begin{definition}\label{def;oasa}
Let $\phi \in \ucpa$ such that $V^* \rho V$ be the minimal Stinespring dilation of $\phi$ with $V : H \rightarrow K$. Let $\clB$ be an abelian subalgebra of $\acom$. Suppose $K$ and $\rho$ disintegrate as $\diX K_p \, \mathrm{d} \nu$ and $\diX \rho_p \, \mathrm{d} \nu$ respectively with respect to the abelian subalgebra $\clB$. Then we say $\clB$ is an orthogonal abelian subalgebra of $\acom$ if;
\begin{enumerate}
\item the operator $V : H \rightarrow K = \diX K_p \, \mathrm{d} \nu$ can be written as $V = \diX V_p \, \mathrm{d} \nu$ where $V_p : H \rightarrow K_p$ is an isometry for almost every $p \in X$;
\item the abelian algebra $\text{L}^\infty(\ucpa, \mu_\clB)$ is isomorphic to an algebra $\text{L}^\infty(X, \nu)$, where $\mu_\clB$ is the pushforward measure defined on $\ucpa$ using the measurable map say $g : X \setminus X_0 \rightarrow \ucpa$ defined as $p \mapsto V_p^* \rho_p V_p$, where $X_0$ is the $\nu$ measure zero set consisting of those $p$ such that $V_p$ is not isometry or $\rho_p$ is not unital.
\end{enumerate} 
\end{definition}

The following remark shows that the Definition \ref{def;oasa} does not depend on the  $\nu$ measure zero set which is being removed from $X.$
\begin{remark}\label{rem;mzs}
The $\nu$ measure zero set $X_0$ in the above definition need not be unique. Suppose $\tilde{X_0}$ is another $\nu$ measure zero set and $\tilde{g} : X \setminus \tilde{X_0} \rightarrow \ucpa$ be the corresponding function with the pushforward measure as $\tilde{\mu_\clB}$. Then for any Borel measurable set say $E$ of $\ucpa$ we have $\mu_\clB(E) = \nu (g^{-1}(E))$ and $\tilde{\mu_\clB}(E) = \nu (\tilde{g}^{-1}(E))$. Since $g = \tilde{g}$ on $X \setminus X_o \cup \tilde{X_0}$, the measure of the set $g^{-1}(E) \setminus \tilde{g}^{-1}(E) \cup \tilde{g}^{-1}(E) \setminus g^{-1}(E)$ is zero. Hence $\mu_\clB(E) = \tilde{\mu_\clB}(E)$ for all Borel measurable subset $E$ of $\ucpa$. This implies $\text{L}^\infty(\ucpa, \mu_\clB)$ is isomorphic to $\text{L}^\infty(\ucpa, \tilde{\mu_\clB})$.
\end{remark}

Now, we prove the main theorem of this section which characterizes orthogonal measures with orthogonal abelian subalgebras.
\begin{theorem} \label{thm;omoas}
Let $\phi \in \ucpa$ and $V^* \rho V$ be the minimal Stinespring dilation of $\phi$. Then there is a one-one correspondence between the following sets:
\begin{enumerate}
\item The orthogonal measures $\mu$ with $\phi$ as its barycenter, that is $\mu \in  \mathcal{O}_\phi(\ucpa)$.
\item The orthogonal abelian subalgebras, $\clB \subseteq \acom$.
\end{enumerate} 
\end{theorem}
\begin{proof}
Let $\mu \in \mathcal{O}_\phi(\ucpa)$. Then from Remark \ref{rem;avna} we get an abelian subalgebra $\clB_\mu$ of $\acom$. We prove that the abelian subalgebra $\clB_\mu$ is orthogonal. We have $\phi = \int_{\ucpa} \phi^\prime \, \mathrm{d} \mu$, where $V^*_{\phi^\prime} \rho_{\phi^\prime} V_{\phi^\prime}$ is the minimal Stinespring dilation of $\phi^\prime$ and $V_{\phi^\prime} : H \rightarrow K_{\phi^\prime}$. By using Corollary \ref{cor;dint} we know $K$ and $\rho$ disintegrate as $\di K_{\phi^\prime} \, \mathrm{d} \mu$ and $\di \rho_{\phi^\prime} \, \mathrm{d} \mu$ respectively with respect to the abelian subalgebra $\clB_\mu$. Now to show that $\clB_\mu$ is orthogonal we prove the first condition in Definition \ref{def;oasa}. We have for all $h \in H$
\begin{equation*}
Vh = \rho(1_A)Vh = \di \rho_{\phi^\prime}(1_A)V_{\phi^\prime}h \, \mathrm{d} \mu = \di V_{\phi^\prime}h \, \mathrm{d} \mu
\end{equation*}
where the second equality follows from the identification of $\rho(1_A)Vh$ with $\rho_\mu(1_A)h^\mu$ using the unitary $U_\mu$. Since each $\phi^\prime$ is unital, $V_{\phi^\prime}$ is an isometry for almost every ${\phi^\prime}$. This proves the first condition in Definition \ref{def;oasa}. For the second condition, observe that the map $g : \ucpa \setminus X_0 \rightarrow \ucpa$ given by $\phi^\prime \mapsto V_{\phi^\prime}^* \rho_{\phi^\prime} V_{\phi^\prime}$ is equal to the identity map on $\ucpa$ in almost everywhere sense. Therefore, the pushforward measure $\mu_{\clB_\mu}$ obtained using $g$ is same as $\mu$. This implies $\text{L}^\infty(\ucpa, \mu_{\clB_\mu})$ is isomorphic to  $\text{L}^\infty(\ucpa, \mu)$. Hence $\clB_\mu$ is an orthogonal abelian subalgebra of $\acom$.

Conversely, suppose $\clB$ is an orthogonal abelian subalgebra of $\acom$. Let $K$ and $\rho$ disintegrates as $\diX K_p \, \mathrm{d} \nu$ and $\diX \rho_p \, \mathrm{d} \nu$ respectively with respect to $\clB$. Since $\clB$ is an orthogonal abelian subalgebra, the operator $V : H \rightarrow K = \diX K_p \, \mathrm{d} \nu$ can be written as $V = \diX V_p \, \mathrm{d} \nu$, where $V_p : H \rightarrow K_p$ is an isometry for almost every $p$. Also, the abelian algebra $\text{L}^\infty(\ucpa, \mu_\clB)$ is isomorphic to $\text{L}^\infty(X, \nu)$, where $\mu_\clB$ is the pushforward measure defined on $\ucpa$ using the measurable map, say $g : X \rightarrow \ucpa$ given by $p \mapsto V_p^* \rho_p V_p$ almost everywhere. Then for all $a \in A$ and $h_1, h_2 \in H$ we get,
\begin{align*}
\la \phi(a)h_1, h_2 \ra &= \la \rho(a)Vh_1, Vh_2 \ra &\\
&= \la \diX \rho_p(a)V_ph_1 \, \mathrm{d} \nu, \diX V_ph_2 \, \mathrm{d} \nu \ra &\\
&= \int_{X} \la V_p^*\rho_p(a)V_ph_1, h_2 \ra \, \mathrm{d} \nu &\\
&= \int_{\ucpa} \la \phi^\prime(a)h_1, h_2 \ra \, \mathrm{d} \mu_\clB.
\end{align*}
The last equality follows because the integral is defined with respect to the pushforward  measure $\mu_\clB$. Since this is true for all $a \in A$ and $h_1, h_2 \in H$, we have $\mu_\clB \in M_\phi(\ucpa)$. Now it is remaining to prove that $\mu_\clB$ is orthogonal. Consider a Borel measurable subset $E$ of $\ucpa$, then 
\begin{align*}
\int_{E} \la \phi^\prime(a)h_1, h_2 \ra \, \mathrm{d} \mu_\clB &= \int_{g^{-1}(E)} \la V_p^*\rho_p(a)V_ph_1, h_2 \ra \, \mathrm{d} \nu &\\
&= \int_{g^{-1}(E)} \la \rho_p(a)V_ph_1, V_ph_2 \ra \, \mathrm{d} \nu &\\
&= \la \chi_{g^{-1}(E)} \left (\diX \rho_p(a)V_ph_1 \ \mathrm{d} \nu \right), \diX V_ph_2 \, \mathrm{d} \nu \ra &\\
&= \la \chi_{g^{-1}(E)} \rho(a)Vh_1, Vh_2 \ra &\\
&= \la V^* \chi_{g^{-1}(E)} \rho(a)Vh_1, h_2 \ra
\end{align*}
where $\chi_{g^{-1}(E)}$ is a characteristic projection with respect to $g^{-1}(E)$ and $\chi_{g^{-1}(E)} \in \clB \subseteq \acom$. Similarly
\begin{equation*}
\int_{E^\complement} \la \phi^\prime(a)h_1, h_2 \ra \, \mathrm{d} \mu_\clB = \la V^* \chi_{g^{-1}(E^\complement)} \rho(a)Vh_1, h_2 \ra
\end{equation*}
with $\chi_{g^{-1}(E^\complement)} = 1_K - \chi_{g^{-1}(E)} \in \clB \subseteq \acom$. The set $E$ was an arbitrary measurable subset of $\ucpa$, and hence from Definition \ref{def;om}, we conclude that the measure $\mu_\clB \in \mathcal{O}_\phi(\ucpa)$. 

We have defined the correspondences $\mu \mapsto \clB_\mu$ and $\clB \mapsto \mu_\clB$. If we start with $\mu \in \mathcal{O}_\phi(\ucpa)$, then from Proposition \ref{prop;om} and Remark \ref{rem;avna} we get that $\clB_\mu$ is isomorphic to $\text{L}^\infty(\ucpa, \mu)$. Then from $\clB_\mu$ we obtain the orthogonal measure, say $\mu_{\clB_\mu}$ using the disintegration of $K$ and $\rho$ with respect to $\clB_\mu$. But here $K$ and $\rho$ disintegrate as $\di K_{\phi^\prime} \, \mathrm{d} \mu$ and $\di \rho_{\phi^\prime} \, \mathrm{d} \mu$ respectively with respect to the abelian subalgebra $\clB_\mu$. Therefore, the measurable map $g : \ucpa \rightarrow \ucpa$ that has been used to obtain the pushforward measure $\mu_{\clB_\mu}$ is just the identity map on $\ucpa$. Therefore, the pushforward measure $\mu_{\clB_\mu}$ is same as the measure $\mu$.

If we start with an orthogonal abelian subalgebra $\clB \subseteq \acom$, then we get an orthogonal measure $\mu_\clB$ using the disintegration of $K$ and $\rho$ with respect to $\clB$. Suppose $K$ and $\rho$ disintegrate as $\diX K_p \, \mathrm{d} \nu$ and $\diX \rho_p \, \mathrm{d} \nu$ respectively  with respect to $\clB$. The pushforward measure $\mu_\clB$ is obtained by using the measurable map $g : X \rightarrow \ucpa$ that is given by $p \mapsto V^*_p \rho_p V_p$. Also, we have $\clB$ is isomorphic to $\text{L}^\infty(\ucpa, \mu_\clB)$. Since $\mu_\clB$ is orthogonal, by using Lemma \ref{lem;kmu} and Proposition \ref{prop;om}, we get the map $k_{\mu_\clB} : \text{L}^\infty(\ucpa, \mu_\clB) \rightarrow \clB_{\mu_\clB}$ given by for $a, b \in A$ and $h_1, h_2 \in H$
\begin{equation*}
\la k_{\mu_\clB}(f) \rho(a) Vh_1, \rho(b) Vh_2 \ra = \int\limits_{\ucpa} f(\phi^\prime) \la \phi^\prime (b^*a)h_1, h_2 \ra  \, \mathrm{d} \mu_\clB.
\end{equation*} 
But here observe, 
\begin{flalign*}
\la k_{\mu_\clB}(f) \rho(a) Vh_1, \rho(b) Vh_2 \ra & = \int\limits_{\ucpa} f(\phi^\prime) \la \phi^\prime (b^*a)h_1, h_2 \ra \mathrm{d} \mu_\clB & \\ 
& = \int\limits_{X} f \circ g(p) \la V^*_p \rho_p(b^*a) V_p h_1, h_2 \ra  \mathrm{d} \nu & \\
& = \la \int\limits_{X} f \circ g(p) \int\limits_{X} \rho_p(a) V_p h_1, \int\limits_{X} \rho_p(b) V_p h_2 \ra.
\end{flalign*}
Thus, we can identify $k_{\mu_\mathcal{B}}(f)$ with $f \circ g \in \text{L}^\infty(X, \nu)$ and hence $k_{\mu_\mathcal{B}}(f) \in \mathcal{B}$ (here, we conclude this by using the unitary from Theorem \ref{thm;dint}). However, the image $k_{\mu_\mathcal{B}}(f) \in \mathcal{B}_{\mu_\mathcal{B}}$. Thus, the two abelian subalgebras $\mathcal{B}$ and $\mathcal{B}_{\mu_\mathcal{B}}$ are equal. Therefore, we have a one-one correspondence between the two sets.
\end{proof}

The following remark clarifies that the construction of an orthogonal measure on $\ucpa$ using an orthogonal abelian subalgebra does not use the second condition in Definition \ref{def;oasa}. If we have an abelian subalgebra of $\acom$ satisfying the first condition in Definition \ref{def;oasa} that is sufficient to construct an orthogonal measure on $\ucpa$. But the corresponding abelian subalgebra with respect to the constructed orthogonal measure need not be the same as the one which we started with. The second condition in Definition \ref{def;oasa} indeed helps in proving the one-one correspondence in Theorem \ref{thm;omoas}.

\begin{remark} \label{rem;soasa}
Let $\clB$ be an abelian subalgebra of $\acom$ such that it satisfies only the first condition in Definition \ref{def;oasa}. That is the operator $V : H \rightarrow K = \diX K_p \, \mathrm{d} \nu$ can be written as $V = \diX V_p \, \mathrm{d} \nu$ where $V_p : H \rightarrow K_p$ is an isometry for almost every $p \in X$. Define a map $g : X \rightarrow \ucpa$ as $p \mapsto V_p^* \rho_p V_p$ except on a measure zero subset of $X$. Remark \ref{rem;mzs} clarifies that the choice of the measure zero set does not make difference. Let $\mu_\clB$ be the pushforward measure defined on $\ucpa$ using the measurable map $g$. Now the same computations given in the proof of Theorem \ref{thm;omoas} prove that $\mu_\clB \in \mathcal{O}_\phi(\ucpa)$. Corresponding to $\mu_\clB$, we have an orthogonal abelian subalgebra $\clB_{\mu_\clB}$. Then we get a map $k_{\mu_\clB} : \text{L}^\infty(\ucpa, \mu_\clB) \rightarrow \clB_{\mu_\clB}$ and from the last part of the proof of Theorem \ref{thm;omoas} we know that
$k_{\mu_\mathcal{B}}(f) \in \mathcal{B}$ for all $f \in \text{L}^\infty(\ucpa, \mu_\mathcal{B})$. Moreover, $k_{\mu_\clB}$ is a *-isomorphism onto $\clB_{\mu_\clB}$ which implies that the abelian subalgebra $\clB_{\mu_\clB}$ is contained in $\clB$ but it need not be equal to $\clB$. This naturally leads to the following definition below.
\end{remark}

\begin{definition} \label{def;soasa}
Let $\phi \in \ucpa$ and $V^* \rho V$ be the minimal Stinespring dilation of $\phi$ with $V : H \rightarrow K$. Let $\clB$ be an abelian subalgebra of $\acom$. Suppose $K$ and $\rho$ disintegrate as $\diX K_p \, \mathrm{d} \nu$ and $\diX \rho_p \, \mathrm{d} \nu$ respectively with respect to the abelian subalgebra $\clB$. Then we say $\clB$ is a sub-orthogonal abelian subalgebra of $\acom$ if the operator $V : H \rightarrow K = \diX K_p \, \mathrm{d} \nu$ can be written as $V = \diX V_p \, \mathrm{d} \nu$, where $V_p : H \rightarrow K_p$ is an isometry for almost every $p \in X$.
\end{definition}

Clearly, every orthogonal abelian subalgebra is sub-orthogonal. Remark \ref{rem;soasa} implies that with respect to a sub-orthogonal abelian subalgebra of $\acom$ we get an orthogonal measure with barycenter $\phi$. However, we do not use sub-orthogonality in the rest of this article.

\section{Examples of Generalized Orthogonal Measures} \label{S4}
In this section, we give some examples of unital completely positive maps with ranges in $M_n(\bbC)$, admitting barycentric orthogonal measures. We recall that for a given C*-algebra $A$, the set $S(A)$ denotes the state space of $A$. Consider the set $S(A)$ with with the weak*-topology. 

\begin{example} \label{ex;nn}
Let $\omega : A \rightarrow \bbC$ be a state. Let $(\pi_\omega, H_\omega, \Omega_\omega)$ be the corresponding GNS triple. For $n \in \bbN$, consider $\phi^n_\omega : M_n(A) \rightarrow M_n(\bbC)$ defined by 
\begin{equation*}
\phi^n_\omega ([a_{i,j}]) := [(\omega(a_{i,j}))_{i,j}].
\end{equation*}
Then $\phi^n_\omega$ is a unital completely positive map and suppose $V^* \rho V$ is the minimal Stinespring dilation of $\phi^n_\omega$, where $V : \bbC^n \rightarrow K$ is an isometry. For each $\omega^\prime \in S(A)$ we have a unital completely positive map $\phi^n_{\omega^\prime} \in \ucpann$ which is defined similarly as above. Define a measurable function $g : S(A) \rightarrow \ucpann$ as:
\begin{equation*}
g(\omega^\prime) : = \phi^n_{\omega^\prime}.
\end{equation*}
Let $\mu$ be a measure on the state space $S(A)$ of $A$ with barycenter $\omega$. Then we have
\begin{equation*}
\phi^n_\omega ([a_{i,j}]) = [(\omega(a_{i,j}))_{i,j}] = \left [ \left (\int_{S(A)} \omega^\prime(a_{i,j}) \, \mathrm{d} \mu \right ) _{i,j}\right ].
\end{equation*}
Suppose $\tilde{\mu}$ is a pushforward measure on $\ucpann$ obtained by the function $g$. Then for $h, k \in \bbC^n$ we get
\begin{equation*}
\la \phi^n_\omega ([a_{i,j}]) h, k \ra = \int\limits_{\ucpann} \la \phi^\prime([a_{i,j}])h, k \ra \, \mathrm{d} \tilde{\mu}.
\end{equation*}
That is $\tilde{\mu} \in M_{\phi^n_\omega}(\ucpann)$. 

Suppose $\mu$ is an orthogonal measure on $S(A)$ with barycenter $\omega$. Then we claim that $\tilde{\mu} \in \mathcal{O}_{\phi^n_\omega}(\ucpann)$. For showing this we use Proposition \ref{prop;om}. We show that the map $k_{\tilde{\mu}} : \text{L}^\infty (\ucpann, \tilde{\mu}) \rightarrow \rho(M_n(A))^\prime$ is a *-isomorphism onto its range. First, we prove that for any projection $f \in \text{L}^\infty (\ucpann, \tilde{\mu})$ the image $k_{\tilde{\mu}}(f)$ is also a projection. Since $\mu$ is an orthogonal measure, we get a *-isomorphism $k_\mu : \text{L}^\infty (S(A), \mu) \rightarrow \pi_\omega(A)^\prime$ onto its range \cite[Proposition 4.1.22]{BR1}. 

Let $f \in \text{L}^\infty (\ucpann, \tilde{\mu})$ be a projection. Then for $[a_{i,j}] \in M_n(A)$ and $h = (h_1, h_2, ..., h_n), ~ k = (k_1, k_2, ..., k_n) \in \bbC^n$ we have
\begin{align*}
\la k_{\tilde{\mu}}(f) \rho([a_{i,j}])Vh, Vk \ra &= \int_{\ucpann} f(\phi^\prime) \la \phi^\prime([a_{i,j}])h, k \ra \, \mathrm{d} \tilde{\mu} &\\
&= \int_{S(A)} f \circ g(\omega^\prime) \la [(\omega^\prime(a_{i,j}))_{i,j}]h, k \ra \, \mathrm{d} \mu &\\
&= \int_{S(A)} f \circ g (\omega^\prime) \sum_{i,j,l,m=1}^{n} \omega^\prime(a_{i,j})h_l\bar{k}_m  \, \mathrm{d} \mu &\\
&= \sum_{i,j,l,m =1}^{n} \la k_\mu(f \circ g)\pi_\omega(a_{i,j})h_l \Omega_\omega, k_m \Omega_\omega \ra &\\
&= \la D [\pi_\omega(a_{i,j})_{i,j}] \begin{bmatrix}
h_1 \Omega_\omega \\
\vdots \\
h_n \Omega_\omega
\end{bmatrix}, \begin{bmatrix}
k_1 \Omega_\omega \\
\vdots \\
k_n \Omega_\omega
\end{bmatrix} \ra
\end{align*}
where
\begin{equation*}
D = \begin{bmatrix}
k_\mu(f\circ g)       &0 &\ldots  &0\\
0 & k_\mu(f\circ g)      &\ddots  &\vdots\\
\vdots  &\ddots  &k_\mu(f\circ g)       &0\\
0 &\ldots  &0 &k_\mu(f\circ g)
\end{bmatrix}.
\end{equation*}

The Hilbert space $K$ is the closed linear span of the elements of the form $\{ \rho([a_{i,j}])Vh ~|~ [a_{i,j}] \in M_n(A), ~ h \in \mathbb{C}^n \}$. Using this, we define a map, say $U^n _\omega : K \rightarrow  \underset{n-times}{\underbrace{H_\omega \oplus \dots \oplus H_\omega}}$ by 
\begin{equation*}
U^n _\omega (\rho([a_{i,j}])Vh) = [\pi_\omega(a_{i,j})_{i,j}][h_i \Omega_\omega].
\end{equation*}  
(Here, we denote an element of the Hilbert space $\underset{n-times}{\underbrace{H_\omega \oplus \dots \oplus H_\omega}}$ as an $n \times 1$ column vector.) First, we check $U^n_\omega$ is an isometry. For this, consider $[a_{i,j}], [b_{i,j}] \in M_n(A)$ and $h, l \in \mathbb{C}^n$. Then, we have,
\begin{align*}
\la \rho([a_{i,j}])Vh, \rho([b_{i,j}])Vl \ra &= \la [a_{i,j}] \otimes h, [b_{i,j}] \otimes l \ra &\\
&= \la \phi^n_\omega( [b_{i,j}]^*[a_{i,j}]) h, l \ra & \\
&= \sum_{i,j,k=1}^{n} \omega(b^*_{i,j}a_{i,k})h_k \bar{l}_j.
\end{align*}
Next, observe 
\begin{align*}
[\pi_\omega(a_{i,j})_{i,j}][h_i \Omega_\omega] = \begin{bmatrix}
\sum\limits_{j=1}^{n} \pi_\omega(a_{1,j})h_j\Omega_\omega \\
\sum\limits_{j=1}^{n} \pi_\omega(a_{2,j})h_j\Omega_\omega \\
\vdots \\
\sum\limits_{j=1}^{n} \pi_\omega(a_{n,j})h_j\Omega_\omega \\
\end{bmatrix} = \begin{bmatrix}
\sum\limits_{j=1}^{n} a_{1,j} \otimes h_j \\
\sum\limits_{j=1}^{n} a_{2,j} \otimes h_j \\
\vdots \\
\sum\limits_{j=1}^{n} a_{n,j} \otimes h_j \\
\end{bmatrix}
\end{align*}
and 
\begin{align*}
[\pi_\omega(b_{i,j})_{i,j}][l_i \Omega_\omega] = \begin{bmatrix}
\sum\limits_{j=1}^{n} \pi_\omega(b_{1,j})l_j\Omega_\omega \\
\sum\limits_{j=1}^{n} \pi_\omega(b_{2,j})l_j\Omega_\omega \\
\vdots \\
\sum\limits_{j=1}^{n} \pi_\omega(b_{n,j})l_j\Omega_\omega \\
\end{bmatrix} = \begin{bmatrix}
\sum\limits_{j=1}^{n} b_{1,j} \otimes l_j \\
\sum\limits_{j=1}^{n} b_{2,j} \otimes l_j \\
\vdots \\
\sum\limits_{j=1}^{n} b_{n,j} \otimes l_j \\
\end{bmatrix}.
\end{align*}
Then
\begin{flalign*}
\la [\pi_\omega(a_{i,j})_{i,j}][h_i \Omega_\omega], [\pi_\omega(b_{i,j})_{i,j}][l_i \Omega_\omega] \ra =& \la \sum\limits_{j=1}^{n} a_{1,j} \otimes h_j, \sum\limits_{j=1}^{n} b_{1,j} \otimes l_j  \ra + \dots   &\\
& + \la \sum\limits_{j=1}^{n} a_{n,j} \otimes h_j, \sum\limits_{j=1}^{n} b_{n,j} \otimes l_j \ra & \\
=& \sum_{i,j,k=1}^{n} \omega(b^*_{i,j}a_{i,k})h_k \bar{l}_j.
\end{flalign*}
Thus, we have shown that
\begin{align*}
\la \rho([a_{i,j}])Vh, \rho([b_{i,j}])Vl \ra = \la [\pi_\omega(a_{i,j})_{i,j}][h_i \Omega_\omega], [\pi_\omega(b_{i,j})_{i,j}][l_i \Omega_\omega] \ra.
\end{align*}
This implies that the map $U^n_\omega$ is an isometry from the Hilbert space $K$ into $\underset{n-times}{\underbrace{H_\omega \oplus \dots \oplus H_\omega}}$. Also, one can note that $U^n_\omega$ is a surjective map onto the Hilbert space $\underset{n-times}{\underbrace{H_\omega \oplus \dots \oplus H_\omega}}$. Therefore, $U^n_\omega$ defines an unitary between the Hilbert spaces $K$ and $\underset{n-times}{\underbrace{H_\omega \oplus \dots \oplus H_\omega}}$.

Using this, we identify $k_{\tilde{\mu}}(f) = D$ and $\rho([a_{i,j}])Vh = [\pi_\omega(a_{i,j})_{i,j}][h_i \Omega_\omega]$. We know that $k_\mu(f \circ g)$ is a projection and hence $k_{\tilde{\mu}}(f)$ is also a projection. 

If $f_1$ and $f_2$ are orthogonal projections in $\text{L}^\infty(\ucpann, \mu)$, then $f_1 \leq 1-f_2$. Since $k_{\tilde{\mu}}$ is a linear positive map we get $k_{\tilde{\mu}}(f_1) \leq 1_K - k_{\tilde{\mu}}(f_2)$ which implies $k_{\tilde{\mu}}(f_1)k_{\tilde{\mu}}(f_2) = 0$. If $f_1$ and $f_2$ are arbitrary projections in $\text{L}^\infty(\ucpann, \mu)$, then repeating the same arguments from the proof of Proposition \ref{prop;om} we get $k_{\tilde{\mu}}(f_1f_2) = k_{\tilde{\mu}}(f_1)k_{\tilde{\mu}}(f_2)$. Again the similar arguments as in the proof of Proposition \ref{prop;om} imply that $k_{\tilde{\mu}}$ is a *-isomorphism onto its range. Therefore $\tilde{\mu}\in \mathcal{O}_{\phi^n_\omega}(\ucpann)$. 
\end{example}

\begin{example} \label{ex;1n}
Let $\omega : A \rightarrow \bbC$ be a state and $(\pi_\omega, H_\omega, \Omega_\omega)$ be the GNS triple corresponding to $\omega$. For $n \in \bbN$, define a map $\phi_{n,\omega} : A \rightarrow M_n(\bbC)$ by 
\begin{equation*}
\phi_{n,\omega} (a) := \begin{bmatrix}
\omega(a)       &0 &\ldots  &0\\
0 & \omega(a)      &\ddots  &\vdots\\
\vdots  &\ddots  & \omega(a)       &0\\
0 &\ldots  &0 & \omega(a)
\end{bmatrix}.
\end{equation*}
Then $\phi_{n,\omega}$ is a unital completely positive map. Suppose $V^* \rho V$ is the minimal Stinespring dilation of $\phi_{n,\omega}$, where $V : \bbC^n \rightarrow K$ is an isometry. Similarly, for each $\omega^\prime \in S(A)$ we have a unital completely positive map $\phi_{n, \omega^\prime} \in \ucpan$ as defined above. Define a measurable function $g : S(A) \rightarrow \ucpan$ as:
\begin{equation*}
g(\omega^\prime) : = \phi_{n, \omega^\prime}.
\end{equation*}
Let $\mu$ be a measure on the state space $S(A)$ of $A$ such that $\omega = \int_{S(A)} \omega^\prime \, \mathrm{d} \mu$. Then we have
\begin{align*}
\phi_{n, \omega} (a) &= \begin{bmatrix}
\omega(a)       &0 &\ldots  &0\\
0 & \omega(a)      &\ddots  &\vdots\\
\vdots  &\ddots  & \omega(a)       &0\\
0 &\ldots  &0 & \omega(a)
\end{bmatrix} \\
&= \begin{bmatrix}
\int\limits_{S(A)} \omega^\prime(a) \, \mathrm{d} \mu      &0 &\ldots  &0\\
0 & \int\limits_{S(A)} \omega^\prime(a) \, \mathrm{d} \mu      &\ddots  &\vdots\\
\vdots  &\ddots  & \int\limits_{S(A)} \omega^\prime(a) \, \mathrm{d} \mu   &0\\
0 &\ldots  &0 & \int\limits_{S(A)} \omega^\prime(a) \, \mathrm{d} \mu
\end{bmatrix}.
\end{align*}
Let $\tilde{\mu}$ be a pushforward measure on $\ucpan$ defined using the function $g$. Then for $h, k \in \bbC^n$ we get
\begin{equation*}
\la \phi_{n, \omega} (a) h, k \ra = \int\limits_{\ucpan} \la \phi^\prime(a)h, k \ra \, \mathrm{d} \tilde{\mu}.
\end{equation*}
That is $\tilde{\mu} \in M_{\phi_{n, \omega}}(\ucpan)$. 

Suppose $\mu$ is an orthogonal measure on $S(A)$ with barycenter $\omega$. Then we claim that $\tilde{\mu} \in \mathcal{O}_{\phi_{n, \omega}}(\ucpan)$. For showing this, we use the same technique from the previous example. We show that the map $k_{\tilde{\mu}} : \text{L}^\infty (\ucpan, \tilde{\mu}) \rightarrow \rho(A)^\prime$ is a *-isomorphism onto its range. Since $\mu$ is an orthogonal measure we get a *-isomorphism $k_\mu : \text{L}^\infty (S(A), \mu) \rightarrow \pi_\omega(A)^\prime$ onto its range \cite[Proposition 4.1.22]{BR1}. 

Let $f \in \text{L}^\infty (\ucpan, \tilde{\mu})$ be a projection. Then for $a \in A$ and $h = (h_1, h_2, ..., h_n), ~ k = (k_1, k_2, ..., k_n) \in \bbC^n$ we have
\begin{flalign*}
& \la k_{\tilde{\mu}}(f) \rho(a)Vh, Vk \ra \\
&= \int_{\ucpan} f(\phi^\prime) \la \phi^\prime(a)h, k \ra \, \mathrm{d} \tilde{\mu} & \\
&=  \int_{S(A)} f \circ g (\omega^\prime) \la \begin{bmatrix}
\omega^\prime(a)       &0 &\ldots  &0\\
0 & \omega^\prime(a)      &\ddots  &\vdots\\
\vdots  &\ddots  & \omega^\prime(a)       &0\\
0 &\ldots  &0 & \omega^\prime(a)
\end{bmatrix} h, k \ra \, \mathrm{d} \mu & \\
&= \int_{S(A)} f \circ g (\omega^\prime) \sum_{i=1}^{n} \omega^\prime(a)h_i\bar{k}_i  \, \mathrm{d} \mu & \\
&= \sum_{i=1}^{n} \la k_\mu(f \circ g)\pi_\omega(a)h_i \Omega_\omega, k_i \Omega_\omega \ra & \\
&= \la D \begin{bmatrix}
\pi_\omega(a)       &0 &\ldots  &0\\
0 & \pi_\omega(a)      &\ddots  &\vdots\\
\vdots  &\ddots  & \pi_\omega(a)      &0\\
0 &\ldots  &0 & \pi_\omega(a)
\end{bmatrix} \begin{bmatrix}
h_1 \Omega_\omega \\
\vdots \\
\vdots \\
h_n \Omega_\omega
\end{bmatrix}, \begin{bmatrix}
k_1 \Omega_\omega \\
\vdots \\
\vdots \\
k_n \Omega_\omega
\end{bmatrix} \ra
\end{flalign*}
where
\begin{equation*}
D = \begin{bmatrix}
k_\mu(f\circ g)       &0 &\ldots  &0\\
0 & k_\mu(f\circ g)      &\ddots  &\vdots\\
\vdots  &\ddots  &k_\mu(f\circ g)       &0\\
0 &\ldots  &0 &k_\mu(f\circ g)
\end{bmatrix}.
\end{equation*}
Now, identify $k_{\tilde{\mu}}(f) = D$ and 
\begin{equation*}
\rho(a)Vh = \begin{bmatrix}
\pi_\omega(a)       &0 &\ldots  &0\\
0 & \pi_\omega(a)      &\ddots  &\vdots\\
\vdots  &\ddots  & \pi_\omega(a)      &0\\
0 &\ldots  &0 & \pi_\omega(a)
\end{bmatrix} \begin{bmatrix}
h_1 \Omega_\omega \\
\vdots \\
\vdots \\
h_n \Omega_\omega
\end{bmatrix}.
\end{equation*}
We know that $k_\mu(f \circ g)$ is a projection and hence $k_{\tilde{\mu}}(f)$ is also a projection. 

Again the similar method as in the previous example shows that  $k_{\tilde{\mu}}$ is a *-isomorphism onto its range. Hence by applying Proposition \ref{prop;om} we conclude that $\tilde{\mu}\in \mathcal{O}_{\phi_{n, \omega}}(\ucpan)$.
\end{example}


\section*{Acknowledgement} 
The authors would like to sincerely thank the referee for a very detailed report and constructive suggestions on the first draft of this article which helped in vastly improving the manuscript.

The first named author is partially supported by Science and Engineering Board (DST, Govt. of India) grant no. ECR/2018/000016 and the second named author is supported by CSIR PhD scholarship with award letter no. 09/1020(0142)/2019-EMR-I.

\subsection*{Declaration}
The authors declare that there is no conflict of interest.

\bibliographystyle{plain}
\bibliography{mybib}   

	






\end{document}